\documentclass[psamsfonts,fceqn,leqno]{amsart}
\usepackage{mathrsfs,latexsym,amsfonts,amssymb,curves,epic}
\usepackage{color}

\newtheorem{theorem}{Theorem}[section]
\newtheorem{corollary}[theorem]{Corollary}
\newtheorem{proposition}[theorem]{Proposition}
\newtheorem{lemma}[theorem]{Lemma}
\newtheorem{question}[theorem]{Question}
\newtheorem{problem}[theorem]{Problem}
\theoremstyle{definition}
\newtheorem{definition}[theorem]{Definition}
\newtheorem{remark}[theorem]{Remark}

\newtheorem{note}[theorem]{Note}

\def\N{{\mathbb N}}

\begin{document}
\title[The $k_{R}$-property of free Abelian topological groups and products of sequential fans]
{The $k_{R}$-property of free Abelian topological groups and products of sequential fans}

  \author{Fucai Lin}
  \address{(Fucai Lin): School of mathematics and statistics,
  Minnan Normal University, Zhangzhou 363000, P. R. China}
  \email{linfucai2008@aliyun.com; linfucai@mnnu.edu.cn}

\author{Shou Lin}
\address{(Shou Lin): School of mathematics and statistics,
  Minnan Normal University, Zhangzhou 363000, P. R. China; Institute of Mathematics, Ningde Teachers' College, Ningde, Fujian
352100, P. R. China} \email{shoulin60@163.com}

  \author{Chuan Liu}
\address{(Chuan Liu): Department of Mathematics,
Ohio University Zanesville Campus, Zanesville, OH 43701, USA}
\email{liuc1@ohio.edu}

  \thanks{The first author is supported by the NSFC (Nos. 11571158, 11471153), the Natural Science Foundation of Fujian Province (Nos. 2017J01405, 2016J05014, 2016J01671, 2016J01672) of China and the Program for New Century Excellent Talents in Fujian Province University.}


  \begin{abstract}
A space $X$ is called a $k_{R}$-space, if
$X$ is Tychonoff and the necessary and sufficient condition for a real-valued
function $f$ on $X$ to be continuous is that the restriction of $f$ to each compact subset is
continuous. In this paper, we discuss the $k_{R}$-property of products of sequential fans and free Abelian topological groups by applying the $\kappa$-fan introduced by Banakh. In particular, we prove the following two results:

(1) The space $S_{\omega_{1}}\times S_{\omega_{1}}$ is not a $k_{R}$-space.

(2)  The space $S_{\omega}\times S_{\omega_{1}}$ is a $k_{R}$-space if and only if $S_{\omega}\times S_{\omega_{1}}$ is a $k$-space if and only if $\mathfrak b>\omega_1$.

These results generalize some well-known results on sequential fans. Furthermore, we generalize some results of Yamada on the free Abelian topological groups by applying the above results. Finally, we pose some open questions about the $k_{R}$-spaces.
  \end{abstract}

  \maketitle

\section{Introduction}
A topological space $X$ is called a {\it $k$-space} if every
subset of $X$, whose intersection with every compact subset $K$ in $X$ is relatively
open in $K$, is open in $X$. It is well-known that the $k$-property which generalizes
metrizability has been studied intensively by topologists and analysts. A space $X$ is called a {\it $k_{R}$-space}, if
$X$ is Tychonoff and the necessary and sufficient condition for a real-valued
function $f$ on $X$ to be continuous is that the restriction of $f$ to each compact subset is
continuous. Clearly every Tychonoff $k$-space is a $k_{R}$-space. The converse
is false. Indeed, for any non-measurable cardinal $\kappa$ the power $\mathbb{R}^{\kappa}$ is a $k_{R}$-space but not a $k$-space, see \cite[Theorem 5.6]{No1970} and \cite[Problem 7.J(b)]{K1975}. Now, the $k_{R}$-property has been widely used in the study of topology, analysis and category, see \cite{BG2014, B2016, Bl1977, B1981, L2006, Lin1991, M1973}.

The results of our research will be presented in two separate papers. In the current paper, we extend some well-known results on $k$-spaces to $k_{R}$-spaces by applying the $\kappa$-fan introduced by Banakh in \cite{B2016}, and then seek some applications in the study of free Abelian topological groups. In the subsequent paper \cite{LLL2016}, we study the $k_{R}$-property in free topological groups.

Let $\kappa$ be an infinite cardinal.
  For each $\alpha\in\kappa$, let $T_{\alpha}$ be a sequence converging to
  $x_{\alpha}\not\in T_{\alpha}$. Let $T=\bigoplus_{\alpha\in\kappa}(T_{\alpha}\cup\{x_{\alpha}\})$ be the topological sum of $\{T_{\alpha}
  \cup \{x_{\alpha}\}: \alpha\in\kappa\}$. Then
  $S_{\kappa}=\{x\}  \cup \bigcup_{\alpha\in\kappa}T_{\alpha}$
  is the quotient space obtained from $T$ by
  identifying all the points $x_{\alpha}\in T$ to the point $x$. The space $S_{\kappa}$ is called a {\it sequential fan}.
Throughout this paper, for convenience we denote $S_{\omega}$ and $S_{\omega_{1}}$ by the following respectively:

$S_{\omega}=\{a_{0}\}\cup\{a(n, m): n, m\in\omega\}$, where for each $n\in\omega$ the sequence $a(n, m)\rightarrow a_{0}$ as $m\rightarrow\infty$;

$S_{\omega_{1}}=\{\infty\}\cup\{x(\alpha, n): n\in\omega, \alpha\in\omega_{1}\}$, where for each $\alpha\in\omega_{1}$ the sequence $x(\alpha, n)\rightarrow \infty$ as $n\rightarrow\infty$.

The paper is organized as follows. In Section 2, we introduce the necessary notation and terminologies which are
  used for the rest of the paper. In Section 3, we investigate the
  $k_{R}$-property of products of sequential fans. First, we prove that $S_{\omega_{1}}\times S_{\omega_{1}}$ is not a $k_{R}$-space, which generalizes a well-konwn result of Gruenhage and Tanaka. Then we prove that $S_{\omega}\times S_{\omega_{1}}$ is a $k_{R}$-space if and only if $S_{\omega}\times S_{\omega_{1}}$ is a $k$-space if and only if $\mathfrak b>\omega_1$. Furthermore, we discuss the topological properties of some class of spaces with the $k_{R}$-property under the assumption of $\mathfrak b\leq\omega_1$.
  Section 4 is devoted to the study of the $k_{R}$-property of free Abelian topological groups. The main theorems in this section generalizes some results in \cite{LL2016} and \cite{Y1993}. In section 5, we pose some questions about $k_{R}$-spaces.

  \section{Preliminaries}
In this section, we introduce the necessary notation and terminologies. Throughout this paper, all topological spaces are assumed to be
  Tychonoff, unless otherwise is explicitly stated.
  First of all, let $\N$ be the set of all positive
  integers and $\omega$ the first infinite ordinal. For a space $X$, we always denote the set of all the non-isolated points by $\mbox{NI}(X)$. For undefined
  notation and terminologies, the reader may refer to \cite{AT2008},
  \cite{E1989}, \cite{Gr1984} and \cite{Lin2015}.

  \medskip
  Let $X$ be a topological space and $A \subseteq X$ be a subset of $X$.
  The \emph{closure} of $A$ in $X$ is denoted by $\overline{A}$ and the
  \emph{diagonal} of $X$ is denoted by $\Delta_{X}$. Moreover, $A$ is called \emph{bounded} if every continuous real-valued
  function $f$ defined on $A$ is bounded.
  Recall that $X$ is said to have a \emph{$G_{\delta}$-diagonal} if $\Delta_{X}$ is a $G_\delta$-set
 in $X \times X$. The space $X$ is called a
  \emph{$k$-space} provided that a subset $C\subseteq X$ is closed in $X$ if
  $C\cap K$ is closed in $K$ for each compact subset $K$ of $X$. If there exists
  a family of countably many compact subsets $\{K_{n}: n\in\mathbb{N}\}$ of
  $X$ such that each subset $F$ of $X$ is closed in $X$ provided that
  $F\cap K_{n}$ is closed in $K_{n}$ for each $n\in\mathbb{N}$, then $X$ is called a \emph{$k_{\omega}$-space}.
A space $X$ is called a $k_{R}$-space, if
$X$ is Tychonoff and the necessary and sufficient condition for a real-valued
function $f$ on $X$ to be continuous is that the restriction of $f$ on each compact subset is
continuous. Note that every $k_{\omega}$-space is a $k$-space and every Tychonoff $k$-space is a $k_{R}$-space. A subset $P$ of $X$ is called a
  \emph{sequential neighborhood} of $x \in X$, if each
  sequence converging to $x$ is eventually in $P$. A subset $U$ of
  $X$ is called \emph{sequentially open} if $U$ is a sequential neighborhood of
  each of its points. A subset $F$ of
  $X$ is called \emph{sequentially closed} if $X\setminus F$ is sequentially open. The space $X$ is called a \emph{sequential  space} if each
  sequentially open subset of $X$ is open. The space $X$ is said to be {\it Fr\'{e}chet-Urysohn} if, for
each $x\in \overline{A}\subset X$, there exists a sequence
$\{x_{n}\}$ in $A$ such that $\{x_{n}\}$ converges to $x$.

A space $X$ is called an \emph{ $S_{2}$}-{space} ({\it Arens' space})  if
$$X=\{\infty\}\cup \{x_{n}: n\in \mathbb{N}\}\cup\{x_{n, m}: m, n\in
\omega\}$$ and the topology is defined as follows: Each
$x_{n, m}$ is isolated; a basic neighborhood of $x_{n}$ is
$\{x_{n}\}\cup\{x_{n, m}: m>k\}$, where $k\in\omega$;
a basic neighborhood of $\infty$ is $$\{\infty\}\cup (\bigcup\{V_{n}:
n>k\ \mbox{for some}\ k\in \omega\}),$$ where $V_{n}$ is a
neighborhood of $x_{n}$ for each $n\in\omega$.

  \begin{definition}\cite{B2016}
  Let $X$ be a topological space.

  $\bullet$ A subset $U$ of $X$ is called {\it $\mathbb{R}$-open} if for each point $x\in U$ there is a continuous function $f: X\rightarrow [0, 1]$ such that $f(x)=1$ and $f(X\setminus U)\subset\{0\}$. It is obvious that each $\mathbb{R}$-open set is open. The converse is true for the open subsets of Tychonoff spaces.

  $\bullet$ A subset $U$ of $X$ is called a {\it functional neighborhood} of a set $A\subset X$ if there is a continuous function $f: X\rightarrow [0, 1]$ such that $f(A)\subset\{1\}$ and $f(X\setminus U)\subset\{0\}$. If $X$ is normal, then each neighborhood of a closed subset $A\subset X$ is functional.
  \end{definition}

  \begin{definition}
  Let $\lambda$ be a cardinal. An indexed family $\{X_{\alpha}\}_{\alpha\in\lambda}$ of subsets of a topological space $X$ is called

 $\bullet$ {\it point-countable} if for any point $x\in X$ the set $\{\alpha\in\lambda: x\in X_{\alpha}\}$ is countable;

 $\bullet$ {\it compact-countable} if for any compact subset $K$ in $X$ the set $\{\alpha\in\lambda: K\cap X_{\alpha}\neq\emptyset\}$ is countable;

  $\bullet$ {\it locally finite} if any point $x\in X$ has a neighborhood $O_{x}\subset X$ such that the set $\{\alpha\in\lambda: O_{x}\cap X_{\alpha}\neq\emptyset\}$ is finite;

  $\bullet$ {\it compact-finite} in $X$ if for each compact subset $K\subset X$ the set $\{\alpha\in\lambda: K\cap X_{\alpha}\neq\emptyset\}$ is finite;

  $\bullet$ {\it strongly compact-finite} \cite{B2016} in $X$ if each set $X_{\alpha}$ has an $\mathbb{R}$-open neighborhood $U_{\alpha}\subset X$ such that the family $\{U_{\alpha}\}_{\alpha\in\lambda}$ is compact-finite;

  $\bullet$ {\it strictly compact-finite} \cite{B2016} in $X$ if each set $X_{\alpha}$ has a functional neighborhood $U_{\alpha}\subset X$ such that the family $\{U_{\alpha}\}_{\alpha\in\lambda}$ is compact-finite.
  \end{definition}

 \begin{definition}\cite{B2016}
  Let $X$ be a topological space and $\lambda$ be a cardinal. An indexed family $\{F_{\alpha}\}_{\alpha\in\lambda}$ of subsets of a topological space $X$ is called a {\it fan} (more precisely, a $\lambda$-fan) in $X$ if this family is compact-finite but not locally finite in $X$. A fan $\{X_{\alpha}\}_{\alpha\in\lambda}$ is called {\it strong} (resp. {\it strict}) if each set $F_{\alpha}$ has a $\mathbb{R}$-open neighborhood (resp. functional neighborhood) $U_{\alpha}\subset X$ such that the family $\{U_{\alpha}\}_{\alpha\in\lambda}$ is compact-finite in $X$.

  If all the sets $F_{\alpha}$ of a $\lambda$-fan $\{F_{\alpha}\}_{\alpha\in\lambda}$ belong to some fixed family $\mathscr{F}$ of subsets of $X$, then the fan will be called an {\it $\mathscr{F}^{\lambda}$-fan}. In particular, if each $F_{\alpha}$ is closed in $X$, then the fan will be called a {\it Cld$^{\lambda}$-fan}.
  \end{definition}

  Clearly, we have the following implications:

  $$\mbox{strict fan}\ \Rightarrow\ \mbox{strong fan}\ \Rightarrow\ \mbox{fan}.$$

  \medskip
  Let $\mathscr P$ be a family of subsets of a space $X$. Then, $\mathscr P$ is called a {\it $k$-network}
   if for every compact subset $K$ of $X$ and an arbitrary open set
  $U$ containing $K$ in $X$ there is a finite subfamily $\mathscr {P}^{\prime}\subseteq
  \mathscr {P}$ such that $K\subseteq \bigcup\mathscr {P}^{\prime}\subseteq U$. Recall that a space $X$ is an \emph{$\aleph$-space} (resp.
  {\it $\aleph_{0}$-space}) if
  $X$ has a $\sigma$-locally finite (resp. countable) $k$-network. Recall that a space $X$ is said to be
  \emph{La\v{s}nev} if it is the continuous closed image of some metric space. We list two well-known facts about the La\v{s}nev spaces as follows.

 \medskip
 {\bf Fact 1:} A La\v{s}nev space is metrizable if it contains no closed copy of $S_\omega$, see \cite{Liu1992}.

 \medskip
 {\bf Fact 2:} A La\v{s}nev space is an $\aleph$-space if it contains no closed copy of $S_{\omega_1}$, see \cite{F1985} and \cite{JY1992}.

 \begin{definition}\cite{C1961}
A topological space $X$ is a {\it stratifiable space} if for each open subset $U$ in $X$, one can assign a sequence $\{U_{n}\}_{n=1}^{\infty}$ of open subsets in $X$ such that the following (a)-(c) hold.

(a) $\overline{U_{n}}\subset U$;

(b) $\bigcup_{n=1}^{\infty}U_{n}=U;$

(c) $U_{n}\subset V_{n}$ whenever $U\subset V$.
\end{definition}

Clearly, each La\v{s}nev space is stratifiable \cite{Gr1984}.

\begin{definition}
Let $X$ be a Tychonoff space. An Abelian topological group $A(X)$ is called {\it the free Abelian topological group} over $X$ if $A(X)$ satisfies the following conditions:

\smallskip
(i) there is a continuous mapping $i: X\rightarrow A(X)$ such that $i(X)$ algebraically generates $A(X)$;

\smallskip
(ii) if $f: X\rightarrow G$ is a continuous mapping to an Abelian topological group $G$, then there exists a continuous homomorphism $\overline{f}: A(X)\rightarrow G$ with $f=\overline{f}\circ i$.
\end{definition}

 \medskip
  Let $X$ be a non-empty Tychonoff space. Throughout this paper, $-X=\{-x: x\in X\}$, which is just a copy of
  $X$. Let $0$ be the neutral element of $A(X)$, that is, the empty
  word. For every $n\in\N$, an element
  $x_{1}+x_{2}+\cdots +x_{n}$ is also called a {\it form} for $(x_{1}, x_{2},
  \cdots, x_{n})\in(X\oplus -X\oplus\{0\})^{n}$. The word $g$ is
  called {\it reduced} if it does not contain any pair of
  symbols of $x$ and $-x$. It follows
  that if the word $g$ is reduced and non-empty, then it is different
  from the neutral element $0$ of $A(X)$. Clearly, each element
  $g\in A(X)$ distinct from the neutral element can be uniquely written
  in the form $g=r_{1}x_{1}+r_{2}x_{2}+\cdots +r_{n}x_{n}$, where
  $n\geq 1$, $r_{i}\in\mathbb{Z}\setminus\{0\}$, $x_{i}\in X$, and
  $x_{i}\neq x_{j}$ for $i\neq j$. In this case, the number $\sum_{i=1}^{n}|r_{i}|$  is said to the {\it reduced length} of $g$ (in particular, the neutral element has the reduced length 0), and the {\it support}
  of $g=r_{1}x_{1}+ r_{2}x_{2}+\cdots +r_{n}x_{n}$ is defined as
  $\mbox{supp}(g) =\{x_{1}, \cdots, x_{n}\}$. Given a subset $K$ of
  $A(X)$, we put $\mbox{supp}(K)=\bigcup_{g\in K}\mbox{supp}(g)$. For every $n\in\mathbb{N}$, let $$i_n: (X\oplus -X
  \oplus\{0\})^{n} \to A_n(X)$$ be the natural mapping defined by
  $$i_n(x_1, x_2, ... x_n)= x_1+x_2+...+x_n$$
  for each $(x_1, x_2, ... x_n) \in (X\oplus -X
  \oplus\{0\})^{n}$.

  Let $X$ be a space. For every $n\in\mathbb{N}$, $A_{n}(X)$ denotes the
  subspace of $A(X)$ that consists of all the words of reduced length at
  most $n$ with respect to the free basis $X$.

 The reader may refer to \cite{Ro1996} for undefined notation
and terminologies of free groups.

  \section{The $k_{R}$-property in sequential fans}
In this section we discuss the $k_{R}$-property of products of sequential fans and generalize some well-known results.
First, we recall a well-known theorem of Gruenhage and Tanaka as follows:

\begin{theorem}\cite[Corollary 7.6.23]{AT2008}
The product $S_{\omega_{1}}\times S_{\omega_{1}}$ is not a $k$-space.
\end{theorem}

Next we generalize this theorem and prove that the product $S_{\omega_{1}}\times S_{\omega_{1}}$ is not a $k_{R}$-space. First of all, we give an important lemma.

\begin{lemma}\label{strict 1}\cite[Proposition 3.2.1]{B2016}
A $k_{R}$-space $X$ contains no strict Cld-fan.
\end{lemma}

Now we can prove one of the main results.

\begin{theorem}\label{l2}
The space $S_{\omega_{1}}\times S_{\omega_{1}}$ is not a $k_{R}$-space.
\end{theorem}

\begin{proof}
By Lemma~\ref{strict 1}, it suffices to prove that $S_{\omega_{1}}\times S_{\omega_{1}}$ contains a strict Cld$^{\omega}$-fan. Next we construct a strict Cld$^{\omega}$-fan in $S_{\omega_{1}}\times S_{\omega_{1}}$.

It follows from \cite[Theorem 20.2]{JW1997} that we can find two families $\mathscr{A}=\{A_{\alpha}: \alpha\in\omega_{1}\}$ and $\mathscr{B}=\{B_{\alpha}: \alpha\in\omega_{1}\}$ of infinite subsets of $\omega$ such that

(a) $A_{\alpha}\cap B_{\beta}$ is finite for all $\alpha, \beta<\omega_{1}$;

(b) for no $A\subset \omega$, all the sets $A_{\alpha}\setminus A$ and $B_{\alpha}\cap A$, $\alpha\in\omega_{1}$ are finite.

For each $n\in\mathbb{N}$, put $$X_{n}=\{(x(\alpha, n), x(\beta, n)): n\in A_{\alpha}\cap B_{\beta}, \alpha, \beta\in\omega_{1}\}.$$ It is obvious that $(\infty, \infty)\not\in \bigcup_{n\in\mathbb{N}}X_{n}$. However, it follows from the proof of \cite[Lemma 7.6.22]{AT2008} that  $(\infty, \infty)\in\overline{\bigcup_{n\in\mathbb{N}}X_{n}}$ but not for any countable subset of $\bigcup_{n\in\mathbb{N}}X_{n}$, which implies that the family $\{X_{n}: \in\mathbb{N}\}$ is not locally finite in $S_{\omega_{1}}\times S_{\omega_{1}}$. Moreover, it is easy to see that each $X_{n}$ is closed and discrete in $S_{\omega_{1}}\times S_{\omega_{1}}$. We claim that the family $\{X_{n}: n\in\mathbb{N}\}$ is compact-finite in $S_{\omega_{1}}\times S_{\omega_{1}}$. Indeed, let $K$ be an any compact subset of $S_{\omega_{1}}\times S_{\omega_{1}}$. It is easy to see that there exists a finite subset $\{\alpha_{i}\in\omega_{1}: i=1, \cdots, m\}$ of $\omega_{1}$ such that $$K\cap\bigcup_{n\in\mathbb{N}}X_{n}\subset (\bigcup_{i=1}^{m}\{x(\alpha_{i}, n): n\in\mathbb{N}\})\times(\bigcup_{i=1}^{m}\{x(\alpha_{i}, n): n\in\mathbb{N}\}).$$ Assume that $K$ intersects infinitely many $X_{n}$. Then there exist $1\leq i, j\leq m$ and an infinite subset $\{n_{k}: k\in \mathbb{N}\}$ in $\mathbb{N}$ such that $$\{x(\alpha_{i}, n_{k}), x(\alpha_{j}, n_{k}): k\in\mathbb{N}\}\subset K\cap (\bigcup_{n\in\mathbb{N}} X_{n}).$$ It is obvious that $$(\infty, \infty)\in\overline{\{x(\alpha_{i}, n_{k}), x(\alpha_{j}, n_{k}): k\in\mathbb{N}\}}.$$ However, since the point $(\infty, \infty)$ does not belong to the closure of any countable subset of $\bigcup_{n\in\mathbb{N}}X_{n}$ in $S_{\omega_{1}}\times S_{\omega_{1}}$, we obtain a contradiction.

Since $X_{n}$ is also an functional neighborhood of itself for each $n\in\mathbb{N}$, the space $S_{\omega_{1}}\times S_{\omega_{1}}$ contains a strict $\mathrm{Cld}^{\omega}$-fan $\{X_{n}: n\in\mathbb{N}\}$.
\end{proof}

Next we prove the second main result in this section. First, we recall some concepts.

Consider $^\omega\omega$, the collection of all functions from $\omega$ to $\omega$. For any $f, g\in ^\omega\omega$, define $f\leq g$ if $f(n)\leq g(n)$ for all but finitely many $n\in \omega$. A subset $\mathscr{H}$ of $^\omega\omega$ is {\it bounded} if there is a $g\in$$^\omega\omega$ such that $f\leq g$ for all $f\in\mathscr{H}$, and is {\it unbounded} otherwise. We denote by $\mathfrak b$ the smallest cardinality of an unbounded family in $^\omega\omega$. It is well-known that $\omega <\mathfrak b\leq \mathrm{c}$, where $\mathrm{c}$ denotes the cardinality of the continuum. Let $\mathscr{F}$ be the set of all finite subsets of $\omega$.

In \cite{Gr1980}, Gruenhage proved that $\omega_1<\mathfrak b$ if and only if the product $S_{\omega}\times S_{\omega_{1}}$ is a $k$-space. Indeed, we have the following result.

\begin{theorem}\label{BF}
The following statements are equivalent.
 \begin{enumerate}
\item $\mathfrak b>\omega_1$.
\item The product $S_{\omega}\times S_{\omega_{1}}$ is a $k_{R}$-space.
\item The product $S_{\omega}\times S_{\omega_{1}}$ is a $k$-space.
  \end{enumerate}
\end{theorem}

\begin{proof}
By Gruenhage's result, it suffices to prove (2) $\Rightarrow$ (1). Assume that $\mathfrak b\leq\omega_1$. Then there exists a subfamily $\{f_{\alpha}\in ^{\omega}\omega: \alpha<\omega_{1}\}$ such that for any $g\in ^{\omega}\omega$ there exists $\alpha<\omega_{1}$ such that $f_{\alpha}(n)>g(n)$ for infinitely many $n$'s. For each $\alpha<\omega_{1}$, put $$G_{\alpha}=\{(a(n, m), x(\alpha, m)): m\leq f_{\alpha}(n), n\in\omega\}.$$ Obviously, each $G_{\alpha}$ is clopen in $S_{\omega}\times S_{\omega_{1}}$. Moreover, it is easy to see that the family $\{G_{\alpha}\}_{\alpha<\omega_{1}}$ of subsets is compact-finite. However, the family $\{G_{\alpha}\}_{\alpha<\omega_{1}}$ is not locally finite at the point $(a_{0}, \infty)$ in $S_{\omega}\times S_{\omega_{1}}$. Indeed, since each $G_{\alpha}$ is clopen in $S_{\omega}\times S_{\omega_{1}}$ and $(a_{0}, \infty)\not\in G_{\alpha}$, it suffices to prove that $(a_{0}, \infty)\in \overline{\bigcup_{\alpha<\omega_{1}}G_{\alpha}}$. Take an arbitrary open neighborhood $U$ at $(a_{0}, \infty)$ in $S_{\omega}\times S_{\omega_{1}}$. Then there exists $f\in ^{\omega}\omega$ and $g\in ^{\omega_{1}}\omega$ such that $$(\{a_{0}\}\cup\{a(n, m): m\geq f(n), n\in\omega\})\times(\{\infty\}\cup\{x(\alpha, m): m\geq g(\alpha), \alpha<\omega_{1}\})\subset U.$$Therefore, there exists $\alpha<\omega_{1}$ such that $f_{\alpha}(n)>f(n)$ for infinitely many $n$'s, then there is a $j\in\omega$ such that $j\geq g(\alpha)$ and $f_{\alpha}(j)>f(j)$, which shows $(a(j, f_{\alpha}(j)), x(\alpha, j))\in U\cap G_{\alpha}$. By the arbitrary choice of $U$, the point $(a_{0}, \infty)$ is in $\overline{\bigcup_{\alpha<\omega_{1}}G_{\alpha}}$.

Since $S_{\omega}\times S_{\omega_{1}}$ is normal and each $G_{\alpha}$ is clopen in $S_{\omega}\times S_{\omega_{1}}$, each $G_{\alpha}$ is also a functional neighborhood of itself. Therefore, it follows from Lemma~\ref{strict 1} that $S_{\omega}\times S_{\omega_{1}}$ is not a $k_{R}$-space, which is a contradiction. Hence $\mathfrak b>\omega_1$.
\end{proof}

 \begin{note}If $\mathfrak b\leq\omega_1$, then some classes of spaces with the $k_{R}$-property have some special topological properties, see the following Theorem~\ref{kr charac}. In order to prove this theorem, we need some concepts and technique lemmas.
Yamada in \cite{Y1993} introduced the following spaces.
\end{note}

\smallskip
Let $\mathfrak{T}$ be a class of metrizable spaces such that each element $P$ of $\mathfrak{T}$ can be represented as $P=X_{0}\cup \bigcup_{i=1}^{\infty}X_{i}$ satisfying the following conditions:

\smallskip
(1) $X_{i}$ is an infinite discrete open subspace of $P$ for every $i\in \mathbb{N}$,

\smallskip
(2) the family $\{X_{i}: i\in\omega\}$ is pairwise disjoint, and

\smallskip
(3) $X_{0}$ is a compact subspace of $P$, and $\{V_{k}=X_{0}\cup \bigcup_{i=k}^{\infty}X_{i}: k\in\mathbb{N}\}$ is a neighborhood base at $X_{0}$ in $P$.

\smallskip
In the above definition, if each $X_{i}$ consists of countably many elements and $X_{0}$ is an one-point set, we denote the space by $P_{0}$. We put $C=\{\frac{1}{n}: n\in\mathbb{N}\}\cup\{0\}$ with the subspace topology of $I$. Let $$\mathbf{G}_{0}=\bigoplus\{C_{i}: i\in\mathbb{N}\}\bigoplus P_{0},$$ where each $C_{i}$ is a copy of $C$ for each $i\in\mathbb{N}$.
Let $$\mathbf{G}_{1}=\bigoplus\{C_\alpha: \alpha<\omega_1\},$$ where $C_\alpha=\{c(\alpha, n): n\in \mathbb{N}\}\cup\{c_\alpha\}$ with $c(\alpha, n)\to c_\alpha$ as $n\rightarrow\infty$ for each $\alpha\in\omega_{1}$.

\smallskip
From here on, we shall use the notations $\mathfrak{T}$, $P_{0}$, $\mathbf{G}_{0}$ and $\mathbf{G}_{1}$ with the meaning
the above meaning.

\begin{lemma}\label{M1}
The space $S_{\omega}\times P$ is not a $k_{R}$-space for each space $P$ in $\mathfrak{T}$.
\end{lemma}

\begin{proof}
Fix an arbitrary space $P$ in $\mathfrak{T}$. Then $P$ can be represented as $X_{0}\cup \bigcup_{i=1}^{\infty}X_{i}$ satisfying (1)-(3) in the above definition. Obviously, $S_{\omega}\times P$ is a normal space. By Lemma~\ref{strict 1}, it suffices to prove that $S_{\omega}\times P$ contains a strict closed$^{\omega}$-fan.

For each $n\in\mathbb{N}$, take an arbitrary countably infinite and pairwise disjoint subset $\{b(n, m): m\in\mathbb{N}\}$ of $X_{n}$. For each $n\in\mathbb{N}$, put $$F_{n}=\{(a(n, m), b(n, m)): m\in\mathbb{N}\}.$$ Obviously, each $F_{n}$ is closed. We claim that the family $\{F_{n}\}$ is compact-finite in $S_{\omega}\times P$. Indeed, for each compact subset $K$ in $S_{\omega}\times P$, there exist compact subsets $K_{1}$ and $K_{2}$ in $S_{\omega}$ and $P$ respectively such that $K\subset K_{1}\times K_{2}$. Since $K_{1}$ is compact in $S_{\omega}$, there exists a natural number $n_{0}$ such that $$K_{1}\subset \{a(n, m): n\leq n_{0_{}}, m\in\omega\}\cup\{a_{0}\}.$$ Therefore, $K\cap F_{n}=\emptyset$ for each $n>n_{0}$.

It is easy to see that $$\emptyset\neq\overline{\bigcup_{n\in\mathbb{N}} F_{n}}\setminus\bigcup_{n\in\mathbb{N}} F_{n}\subset\{a_{0}\}\times X_{0}.$$ Hence the family $\{F_{n}\}$ is not locally finite at $S_{\omega}\times P$. Since $S_{\omega}\times P$ is a normal $\aleph$-space, it follows from \cite[Proposition 2.9.2]{B2016} that the family $\{F_{n}\}$ is strongly compact-finite, and then by the normality the family $\{F_{n}\}$ is also strictly compact-finite.

Therefore, the family $\{F_{n}\}$ is a strict closed$^{\omega}$-fan in $S_{\omega}\times P$.
\end{proof}

\begin{proposition}\label{network}
Let $S_{\omega}\times Z$ be a $k_{R}$-space, where $Z$ is a stratifiable space. If $\{Z_{n}\}$ is a decreasing network for some point $z_{0}$ in $Z$, then the closure of $Z_{n}$ is compact for some $n$.
\end{proposition}

\begin{proof}
If not, then each $\overline{Z_{n}}$ is not compact and thus not countably compact. Hence there exist a sequence $\{n_{k}\}$ in $\mathbb{N}$ and the family $\{C_{k}\}$ of countably infinite, discrete and closed subsets of $Z$ such that $C_{k}\subset \overline{Z_{n_{k}}}\setminus\{z_{0}\}$ for each $k\in\mathbb{N}$, and the family $\{C_{k}\}$ is pairwise disjoint. Put $$Z_{0}=\{z_{0}\}\cup\bigcup_{k\in\mathbb{N}} C_{k}.$$ It is easy to see that $Z_{0}$ belongs to $\mathfrak{T}$ and is closed in $Z$. Since $Z$ is stratifiable, $S_{\omega}\times Z$ is stratifiable, hence $S_{\omega}\times Z_{0}$ is stratifiable. By \cite[Proposition 5.10]{BG2014}, $S_{\omega}\times Z_{0}$ is a $k_{R}$-subspace by the assumption. However, it follows from Lemma~\ref{M1} that $S_{\omega}\times Z_{0}$ is not a $k_{R}$-subspace, which is a contradiction. Therefore, there exists $n\in\mathbb{N}$ such that the closure of $Z_{n}$ is compact.
\end{proof}

By Proposition~\ref{network}, we have the following corollary.

\begin{corollary}\label{compact}
Let $X$ be a stratifiable space with a point-countable $k$-network. If $S_{\omega}\times X$ is a $k_{R}$-space, then $X$ has a point-countable $k$-network consisting of sets with compact closures.
\end{corollary}

Let $X$ and $Y$ be two topological spaces. We recall that the pair $(X, Y)$ satisfies the {\it Tanaka's conditions} \cite{T1976} if one of the following (1)-(3) hold:

\smallskip
(1) Both $X$ and $Y$ are first-countable;

\smallskip
(2) Both $X$ and $Y$ are $k$-spaces, and $X$ or $Y$ is locally compact\footnote{A space $X$ is called {\it locally compact} if every point of $X$ has a compact neighbourhood.};

\smallskip
(3) Both $X$ and $Y$ are locally $k_{\omega}$-spaces\footnote{A space $X$ is called {\it locally $k_{\omega}$} if every point of $X$ has a $k_{\omega}$-neighbourhood.}.

\begin{theorem}\label{kr charac}
Assume $\mathfrak b\leq\omega_1$, then the following statements hold.
 \begin{enumerate}
\item If $X$ is a stratifiable $k$-space with a point-countable $k$-network and $S_{\omega}\times X$ is a $k_{R}$-space, then $X$ is a locally $\sigma$-compact space; in particular, if $X$ has a compact-countable $k$-network then $X$ is a locally $k_{\omega}$-space.
\item If $X$ is a stratifiable $k$-space with a point-countable $k$-network and $S_{\omega_{1}}\times X$ is a $k_{R}$-space, then $X$ is a locally compact space.
\item If both $X$ and $Y$ are stratifiable $k$-spaces with a compact countable $k$-network, then $X\times Y$ is a $k_{R}$-space if and only if the pair $(X, Y)$ satisfies the conditions of Tanaka.
  \end{enumerate}
\end{theorem}

\begin{proof}
By Theorem~\ref{BF}, we see that $S_{\omega}\times S_{\omega_{1}}$ is not a $k_{R}$-space.

\smallskip
(1) Assume that $X$ is a stratifiable $k$-space and $S_{\omega}\times X$ is a $k_{R}$-space. By Corollary~\ref{compact}, it follows that $X$ has a point-countable $k$-network consisting of sets with compact closures. We claim that there is no closed copy $H$ of $X$ such that $H$ is the inverse image of $S_{\omega_{1}}$ under some perfect mapping. Assume to the contrary that there exist a closed subspace $H$ of $X$ and a perfect mapping $f: H\rightarrow S_{\omega_{1}}$ from $H$ onto $S_{\omega_{1}}$. Then it is easy to see that the mapping $\mbox{id}_{S_{\omega}}\times f: S_{\omega}\times H\rightarrow S_{\omega}\times S_{\omega_{1}}$ is also an onto perfect mapping. Moreover, $S_{\omega}\times H$ is a $k_{R}$-space since $S_{\omega}\times X$ is a stratifiable $k_{R}$-space. Since the $k_{R}$-property is preserved by the perfect mappings, $S_{\omega}\times S_{\omega_{1}}$ is a $k_{R}$-space, which is a contradiction.

Since $X$ is a $k$-space with a point-countable $k$-network consisting of sets with compact closures, it follows from \cite[Lemma 1.3]{LiuLin2006} that $X$ contains no closed copy subspace of $X$ such that it is the inverse image of $S_{\omega_{1}}$ under some perfect mapping if and only if $X$ is a locally $\sigma$-compact space, thus $X$ is a locally $\sigma$-compact space.

If $X$ has a compact-countable $k$-network, then $X$ is a locally $\aleph_{0}$-space. Then $X$ is a locally $k_{\omega}$-space since $X$ is a $k$-space.

\smallskip
(2) Assume that $X$ is a stratifiable $k$-space and $S_{\omega_{1}}\times X$ is a $k_{R}$-space. Moreover, since $S_{\omega}\times S_{\omega_{1}}$ is not a $k_{R}$-space, it is easy to see that $X$ contain no closed copies of $S_{2}$ and $S_{\omega}$. Then $X$ is first-countable by \cite[Corollary 3.9]{Lin1997} since $X$ is a $k$-space with a point-countable $k$-network. Since $S_{\omega}\times X$ is a closed subspace of the stratifiable space $S_{\omega_{1}}\times X$, the subspace $S_{\omega}\times X$ is a $k_{R}$-space. By Corollary~\ref{compact}, it follows that $X$ has a point-countable $k$-network consisting of sets with compact closures. Then it follows from \cite[Lemma 2.1]{LiuLin1997} and the first-countability of $X$ that $X$ is locally compact.

\smallskip
(3) Assume that both $X$ and $Y$ are stratifiable $k$-spaces with a compact-countable $k$-network. Obviously, it suffices to prove that if $X\times Y$ is a $k_{R}$-space then $(X, Y)$ satisfies the conditions of Tanaka. We divide the proof into the following cases.

\smallskip
{\bf Case 1} Both $X$ and $Y$ contain no closed copies of $S_{2}$ and $S_{\omega}$.

\smallskip
Since both $X$ and $Y$ are $k$-spaces, it follows from \cite[Corollary 3.9]{Lin1997} that $X$ and $Y$ are all first-countable. Therefore, the pair $(X, Y)$ satisfies the conditions of Tanaka.

\smallskip
{\bf Case 2} Both $X$ and $Y$ contain closed copies of $S_{2}$ or $S_{\omega}$.

\smallskip
Then both $S_{\omega}\times Y$ and $X\times S_{\omega}$ are locally $k_{\omega}$-spaces by (1). Therefore, the pair $(X, Y)$ satisfies the conditions of Tanaka.

\smallskip
{\bf Case 3} The space $X$ contains a closed copy of $S_{2}$ or $S_{\omega}$ and $Y$ contains no closed copies of $S_{2}$ and $S_{\omega}$.

\smallskip
Then $Y$ is first-countable by \cite[Corollary 3.9]{Lin1997} and $S_{\omega}\times Y$ is a $k_{R}$-space since $X\times Y$ is stratifiable. By Corollary~\ref{compact}, it follows that $Y$ has a compact-countable $k$-network consisting of sets with compact closures. Then it follows from \cite[Lemma 2.1]{LiuLin1997} and the first-countability of $Y$ that $Y$ is locally compact. Therefore, the pair $(X, Y)$ satisfies the conditions of Tanaka.

\smallskip
{\bf Case 4} The space $Y$ contains a closed copy of $S_{2}$ or $S_{\omega}$ and $X$ contains no closed copies of $S_{2}$ and $S_{\omega}$.

\smallskip
The proof is the same as the proof of Case 3.
\end{proof}

\begin{corollary}\label{kr charac1}
Assume $\mathfrak b\leq\omega_1$, then the following statements hold.
 \begin{enumerate}
\item If $X$ is a La\v{s}nev space and $S_{\omega}\times X$ is a $k_{R}$-space, then $X$ is a locally $k_{\omega}$-space.
\item If $X$ is a La\v{s}nev space and $S_{\omega_{1}}\times X$ is a $k_{R}$-space, then $X$ is a locally compact space.
\item If $X$ and $Y$ are La\v{s}nev spaces, then $X\times Y$ is a $k_{R}$-space if and only if the pair $(X, Y)$ satisfies the conditions of Tanaka.
  \end{enumerate}
\end{corollary}

\begin{theorem}\label{La}
Let $X$ be a stratifiable space such that $X^2$ is a $k_R$-space. If $X$ satisfies one of the following conditions, then either $X$ is metrizable or $X$ is the topological sum of $k_\omega$-subspaces.
\begin{enumerate}
\item $X$ is a $k$-space with a compact-countable $k$-network;
\item $X$ is a Fr\'echet-Urysohn space with a point-countable $k$-network.
\end{enumerate}
\end{theorem}

\begin{proof} First, suppose that $X$ is a $k$-space with a compact-countable $k$-network. We divide the proof into the following three cases.

 \smallskip
{\bf Case 1.1:} The space $X$ contains a closed copy of $S_\omega$.

 \smallskip
Since $S_\omega\times X$ is a closed subspace of the stratifiable space $X^2$, it follows from \cite[Proposition 5.10]{BG2014} that $S_\omega\times X$ is a $k_R$-space. By Proposition~\ref{network}, the space $X$ has a compact-countable $k$-network consisting of sets with
compact closures $\mathscr{P}$. Then $\mathscr{P}$ is star-countable, hence it follows from \cite{H1964} that we have $$\mathscr{P}=\bigcup_{\alpha\in A}\mathscr{P}_\alpha,$$ where each $\mathscr{P}_\alpha$ is countable and $(\bigcup\mathscr{P}_\alpha)\cap(\bigcup\mathscr{P}_\beta)=\varnothing$ for any $\alpha\neq\beta\in A$. For each $\alpha\in A$, put $X_\alpha=\bigcup\mathscr{P}_\alpha$. Since $\mathscr{P}$ is a $k$-network, it is easy to see that the family $\{X_{\alpha}: \alpha\in A\}$ is compact-finite in $X$. Put $\overline{\mathscr{P}}=\{\overline{P}: P\in\mathscr{P}\}$. We claim that $\overline{\mathscr{P}}$ is star-countable, hence $\overline{\mathscr{P}}$ is compact-countable since $\overline{\mathscr{P}}$ is a $k$-network in $X$.

Suppose not, there exists a $P\in\mathscr{P}$ and an uncountable subfamily $\{P_{\alpha}: \alpha<\omega_{1}\}$ of $\mathscr{P}$ such that $\overline{P}\cap \overline{P_{\alpha}}\neq\emptyset$ for each $\alpha<\omega_{1}$. Without loss of generality, we may assume that $P_{\alpha}\in\mathscr{P}_{\alpha}$ for each $\alpha<\omega_{1}$. Since each $\overline{P_{\alpha}}$ is metrizable, there exists a non-trivial sequence $T_{\alpha}$ in $P_{\alpha}$ converging to some point in $\overline{P}$. Without loss of generality, we may assume that $\overline{P}\cap T_{\alpha}=\emptyset$ for each $\alpha<\omega_{1}$. Let $F=\overline{P}\cup\bigcup_{\alpha<\omega_{1}}T_{\alpha}$. Since $X$ is a $k$-space and the family $\{T_{\alpha}: \alpha<\omega_{1}\}$ is compact-finite, the set $F$ is closed in $X$. Let $f: F\rightarrow F/P$ be the natural quotient mapping. Then $f$ is perfect and $F/P$ is homeomorphic to $S_{\omega_{1}}$, hence $F^{2}$ is the inverse image of $(S_{\omega_{1}})^{2}$ under a perfect mapping. By Theorem~\ref{l2}, $(S_{\omega_{1}})^{2}$ is not a $k_{R}$-space, thus $F^{2}$ is not a $k_{R}$-space. However, since $X$ is a stratifiable space and $F$ is closed in $X$, the subspace $F^{2}$ is a $k_{R}$-space, which is a contradiction.

Therefore, without loss of generality, we may assume that $\mathscr{P}$ is a compact-countable compact $k$-network of $X$. Obviously, each $X_\alpha$ is a closed $k$-subspace of $X$ and has a countable compact $k$-network $\mathscr{P}_\alpha$. Moreover, we claim that each $X_\alpha$ is open in $X$. Indeed, fix an arbitrary $\alpha\in A$. Since $X$ is a $k$-space, it suffices to prove that $\bigcup\{X_{\beta}: \beta\in A, \beta\neq\alpha\}\cap K$ is closed in $K$ for each compact subset $K$ in $X$. Take an arbitrary compact subset $K$ in $X$. Since $\mathscr{P}$ is a $k$-network of $X$, there exists a finite subfamily $\mathscr{P}^{\prime}\subset\mathscr{P}$ such that $K\subset \bigcup\mathscr{P}^{\prime}$. Then
\begin{eqnarray}
\bigcup\{X_{\beta}: \beta\in A, \beta\neq\alpha\}\cap K&=&\bigcup\{X_{\beta}: \beta\in A, \beta\neq\alpha\}\cap K\cap\bigcup\mathscr{P}^{\prime}\nonumber\\
&=&K\cap\{P: P\in\mathscr{P}^{\prime}, P\not\in\mathscr{P}_{\alpha}\}.\nonumber
\end{eqnarray}
Since each element of $\mathscr{P}^{\prime}$ is compact, the set $\bigcup\{X_{\beta}: \beta\in A, \beta\neq\alpha\}\cap K$ is closed in $K$.
Therefore, $X=\bigoplus_{\alpha\in A}X_\alpha$ and each $X_\alpha$ is a $k_\omega$-subspace of $X$. Thus $X$ is the topological sum of $k_\omega$-subspaces.

 \smallskip
{\bf Case 1.2:} The space $X$ contains a closed copy of $S_2$.

 \smallskip
Obviously, $S_2\times X$ is a $k_R$-space. Since $S_\omega$ is the image of $S_2$ under the perfect mapping and the $k_R$-property is preserved by the quotient mapping, $S_\omega\times X$ is a $k_R$-space. By Case 1.1, $X$ is the topological sum of $k_\omega$-subspaces.

 \smallskip
{\bf Case 1.3:} The space $X$ contains no copy of $S_\omega$ or $S_2$.

 \smallskip
Since $X$ is a $k$-space with a point-countable $k$-network, it follows from \cite[Lemma 8]{YanLin1999} and \cite[Corollary 3.10]{Lin1997} that $X$ has a point-countable base, and thus $X$ is metrizable since a stratifiable space with a point-countable base is metrizable \cite{Gr1984}.

Finally, let $X$ be a Fr\'echet-Urysohn space with a point-countable $k$-network. We divide the proof into the following two cases

 \smallskip
{\bf Case 2.1:} The space $X$ contains a closed copy of $S_\omega$ or $S_2$.

 \smallskip
By Case 1.2, without loss of generality we may assume that $X$ contains a closed copy of $S_\omega$. By Proposition~\ref{network}, the space $X$ has a point-countable $k$-network consisting of sets with
compact closures. Since $X$ is a regular Fr\'echet-Urysohn space, it follows from \cite[Corollary 3.6]{Sa97a} that $X$ is a La\v{s}nev space, hence $X$ has a compact-countable $k$-network. By Case 1.1, the space $X$ is the topological sum of $k_\omega$-subspaces.

 \smallskip
{\bf Case 2.2} The space $X$ contains no copy of $S_\omega$ or $S_2$.

 \smallskip
It follows from \cite[Lemma 8]{YanLin1999} and \cite[Corollary 3.10]{Lin1997} that $X$ has a point-countable base. Then $X$ is metrizable since a stratifiable space with a point-countable base is metrizable \cite{Gr1984}.
\end{proof}

\begin{remark}
 By the proof of Theorem~\ref{La}, the condition (2) implies (1) in Theorem~\ref{La}, and furthermore $X^2$ is a $k$-space.
\end{remark}

\maketitle
\section{The applications to free Abelian topological groups}
In this section, we mainly discuss the $k_{R}$-property in the free Abelian topological groups. Recently, T. Banakh in \cite{B2016} proved that $A(X)$ is a $k$-space if $A(X)$ is a $k_{R}$-space for a La\v{s}nev space $X$. Indeed, he obtained this result in wider classes of spaces. However, he did not discuss the following quesiton:

\begin{question}
Let $X$ be a space. For some $n\in\omega$, if $A_{n}(X)$ is a $k_{R}$-space, is $A_{n}(X)$ a $k$-space?
\end{question}

We shall give some answers to the above question and generalize some results of Yamada in the free Abelian topological groups. First, we give a characterization for some class of spaces such that $A_2(X)$ is a $k_{R}$-space if and only if $A_2(X)$ is a $k$-space.

\begin{theorem}\label{A2}
Let $X$ be a stratifiable Fr\'{e}chet-Urysohn space with a point-countable $k$-network. Then the following statements are equivalent:
\begin{enumerate}
\item $A_2(X)$ is a $k$-space;
\item $A_2(X)$ is a $k_R$-space;
\item the space $X$ is metrizable or $X$ is a locally $k_\omega$-space.
\end{enumerate}
\end{theorem}

\begin{proof}
Obviously, we have (1) $\Rightarrow$ (2). It suffices to prove (3) $\Rightarrow$ (1) and (2) $\Rightarrow$ (3).

\smallskip
(3) $\Rightarrow$ (1). Since $(X\cup\{0\}\cup (-X))^2$ is a $k$-space and $i_2$ is a closed mapping, $A_2(X)$ is a $k$-space.

\smallskip
(2) $\Rightarrow$ (3). Use the same notations as in Theorem~\ref{l2}. First we claim that $X$ contains no copy of $S_{\omega_1}$. If not, let $Y_n$=$i_2(X_n)$ for each $n\in\mathbb{N}$. Since the family $\{X_n\}$ is a strict Cld$^\omega$-fan, it follows from \cite[Proposition 3.4.3]{B2016} that the family $\{Y_n\}$ is a strict Cld$^\omega$-fan, which implies that $A_2(X)$ is not a $k_R$-space, a contradiction. Therefore, $X$ contains no closed copy of $S_{\omega_1}$.

\smallskip
If $X$ contains no close copy of $S_\omega$, then it follows from \cite[Lemma 8]{YanLin1999} and \cite[Corollaries 3.9 and 3.10]{Lin1997} that $X$ has a point-countable base, thus it is metrizable since $X$ is stratifiable \cite{Gr1984}. Therefore, we may assume that $X$ contains a closed copy of $S_{\omega}$. Next we prove that $X$ is a locally $k_\omega$-space. Indeed we claim that $X$ contains no closed subspace belonging to $\mathfrak{T}$. Assume to the contrary that $X$ contains a closed subspace $P\in\mathfrak{T}$. By the proof of Lemma~\ref{M1}, $X\times X$ contains a closed Cld$^{\omega}$-fan. Since $X$ is an $\aleph$-space, it follows from \cite[Proposition 2.9.2]{B2016} and the normality of $(X\cup\{0\}\cup (-X))^2$ that $(X\cup\{0\}\cup (-X))^2$ contains a strict Cld$^{\omega}$-fan. Since $i_2$ is a closed mapping, it follows from \cite[Proposition 3.4.3]{B2016} that $A_2(X)$ contains a strict Cld$^{\omega}$-fan, which is a contradiction. Hence $X$ contains no closed subspace belonging to $\mathfrak{T}$, which means that every first-countable subspace is locally compact. Then it follows from \cite[Lemma 2.1]{LiuLin1997} that $X$ has a point-countable $k$-network whose elements have compact closures. Finally, since $X$ is a Fr\'{e}chet-Urysohn space with a point-countable $k$-network, it follows from \cite[Corollary 2.12]{LiuLin2006} or \cite[Corollary 5.4.10]{Lin2015} that $X$ is a locally $k_\omega$-space if and only if $X$ contains no closed copy of $S_{\omega_1}$, thus $X$ is a locally $k_\omega$-space.
\end{proof}

\begin{note}
By Theorem~\ref{A2}, it follows that $A_{2}(S_{\omega_{1}})$ is not a $k_{R}$-space.
\end{note}

 \begin{lemma}\label{k-R and k}
 Let $A(X)$ be a $k_{R}$-space. If each $A_{n}(X)$ is a normal $k$-space, then $A(X)$ is $k$-space.
 \end{lemma}

 \begin{proof}
 It is well-known that each compact subset of $A(X)$ is contained in some $A_{n}(X)$ \cite[Corollary 7.4.4]{AT2008}. Hence it follows from \cite[Lemma 2]{L2006} that $A(X)$ is a $k$-space.
 \end{proof}

\begin{theorem}
Let $X$ be a paracompact $\sigma$-space\footnote{A regular space $X$ is called a {\it $\sigma$-space} if it has a $\sigma$-locally finite network.}. Then $A(X)$ is a $k_{R}$-space and each $A_{n}(X)$ is a $k$-space if and only if $A(X)$ is a $k$-space.
\end{theorem}

\begin{proof}
Since $X$ is a paracompact $\sigma$-space, it follows from \cite[Theorem 7.6.7]{AT2008} that $A(X)$ is also a paracompact $\sigma$-space, hence each $A_{n}(X)$ is normal. Now apply Lemma~\ref{k-R and k} to conclude the proof.
\end{proof}

In \cite{Y1993}, Yamada proved that $A_{4}(\mathbf{G}_{1})$ is not a $k$-space. Indeed, we prove that $A_{4}(\mathbf{G}_{1})$ is not a $k_{R}$-space.

Suppose that $\mathscr{U}_{X}$ is the universal uniformity of a space $X$. Fix an arbitrary $n\in\mathbb{N}$. For each $U\in\mathscr{U}_{X}$ let
$$W_{n}(U)=\{x_{1}-y_{1}+x_{2}-y_{2}+\cdots+x_{k}-y_{k}: (x_{i}, y_{i})\in U\ \mbox{for}\ i=1, \cdots, k, k\leq n\},$$and $\mathscr{W}_{n}=\{W_{n}(U): U\in\mathscr{U}_{X}\}$. Then the family $\mathscr{W}_{n}$ is a neighborhood base at $0$ in $A_{2n}(X)$ for each $n\in\mathbb{N}$, see \cite{Y1993, Y1997}.

\begin{proposition}\label{M3}
The subspace $A_{4}(\mathbf{G}_{1})$ is not a $k_{R}$-space.
\end{proposition}

\begin{proof}
It suffices to prove that $A_{4}(\mathbf{G}_{1})$ contains a strict Cld$^{\omega}$-fan. Let $\mathscr{A}=\{A_{\alpha}: \alpha\in\omega_{1}\}$ and $\mathscr{B}=\{B_{\alpha}: \alpha\in\omega_{1}\}$ be two families of infinite subsets of $\omega$ as in the proof of Theorem~\ref{l2}. For each $n\in\mathbb{N}$, put $$X_{n}=\{c(\alpha, n)-c_{\alpha}+c(\beta, n)-c_{\beta}: n\in A_{\alpha}\cap B_{\beta}, \alpha, \beta\in\omega_{1}\}.$$ It suffices to prove the following three statements.

\smallskip
(1) The family $\{X_{n}\}$ is strictly compact-finite in $A_{4}(\mathbf{G}_{1})$.

\smallskip
Since $\mathbf{G}_{1}$ is a La\v{s}nev space, it follows from \cite[Theorem 7.6.7]{AT2008} that $A(\mathbf{G}_{1})$ is also a paracompact $\sigma$-space, hence $A_{4}(\mathbf{G}_{1})$ is paracompact (and thus normal). Hence it suffices to prove that the family $\{X_{n}\}$ is strongly compact-finite in $A_{4}(\mathbf{G}_{1})$. For each $\alpha\in\omega_{1}$ and $n\in\mathbb{N}$, let $C_{\alpha}^{n}=C_{\alpha}\setminus\{c(\alpha, m): m\leq n\},$ and put $$U_{n}=\{c(\alpha, n)-x+c(\beta, n)-y: n\in A_{\alpha}\cap B_{\beta}, \alpha, \beta\in\omega_{1}, x\in C_{\alpha}^{n}, y\in C_{\beta}^{n}\}.$$
Obviously, each $X_{n}\subset A_{4}(\mathbf{G}_{1})\setminus A_{3}(\mathbf{G}_{1})$. Since $A_{4}(\mathbf{G}_{1})\setminus A_{3}(\mathbf{G}_{1})$ is open in $A_{4}(\mathbf{G}_{1})$, it follows from \cite[Corollary 7.1.19]{AT2008} that each $U_{n}$ is open in $A_{4}(\mathbf{G}_{1})$. We claim that the family $\{U_{n}\}$ is compact-finite in $A_{4}(\mathbf{G}_{1})$.
If not, then there exist a compact subset $K$ in $A_{4}(\mathbf{G}_{1})$ and a subsequence $\{n_{k}\}$ in $\mathbb{N}$ such that $K\cap U_{n_{k}}\neq\emptyset$ for each $k\in\mathbb{N}$. For each $k\in\mathbb{N}$, choose an arbitrary point $$z_{k}=c(\alpha_{k}, n_{k})-x_{k}+c(\beta_{k}, n_{k})-y_{k}\in K\cap U_{n_{k}},$$ where $x_{k}\in C_{\alpha_{k}}^{n_{k}}$ and $y_{k}\in C_{\beta_{k}}^{n_{k}}$. Since $A_{4}(\mathbf{G}_{1})$ is paracompact, it follows from \cite{AOP1989} that the closure of the set supp($K$) is compact in $\mathbf{G}_{1}$. Therefore, there exists $N\in\mathbb{N}$ such that $$\mbox{supp}(K)\cap \bigcup \{C_{\alpha}: \alpha\in\omega_{1}\setminus\{\gamma_{i}\in\omega_{1}: i\leq N\}\}=\emptyset,$$ that is, supp($K$$)\subset \bigcup_{\alpha\in\{\gamma_{i}\in\omega_{1}: i\leq N\}} C_{\alpha}.$ Since each $z_{k}\in K$, there exists $$\alpha_{k}, \beta_{k}\in\{\gamma_{i}\in\omega_{1}: i\leq N\}$$ such that $A_{\alpha_{k}}\cap B_{\beta_{k}}$ is an infinite set, which is a contradiction since $A_{\alpha}\cap B_{\beta}$ is finite for all $\alpha, \beta<\omega_{1}$.

\smallskip
(2) Each $X_{n}$ is closed in $A_{4}(\mathbf{G}_{1})$.

\smallskip
Fix an arbitrary $n\in\mathbb{N}$. Next we prove that $X_{n}$ is closed in $A_{4}(\mathbf{G}_{1})$. Let $Z=\mbox{supp}(X_{n})$. The set $Z$ is a closed discrete subset of $\mathbf{G}_{1}$. Since $\mathbf{G}_{1}$ is metriable, it follows from \cite{U1991} that $A(Z)$ is topologically isomorphic to a closed subgroup of $A(\mathbf{G}_{1})$, hence $A_{4}(Z)$ is a closed subspace of $A_{4}(\mathbf{G}_{1})$. Since $A(Z)$ is discrete and $X_{n}\subset A_{4}(Z)$, the set $X_{n}$ is closed in $A_{4}(Z)$ (and thus closed in $A_{4}(\mathbf{G}_{1})$).

\smallskip
(3) The family $\{X_{n}\}$ is not locally finite at the point $0$ in $A_{4}(\mathbf{G}_{1})$.

\smallskip
Indeed, it suffices to prove that $0\in\overline{\bigcup_{n\in\mathbb{N}}X_{n}}\setminus\bigcup_{n\in\mathbb{N}}X_{n}$. Take an arbitrary $U$ that belongs to the universal uniformity on $\mathbf{G}_{1}$. We shall prove $W_{2}(U)\cap \bigcup_{n\in\mathbb{N}}X_{n}\neq\emptyset$.

\smallskip
Indeed, we can choose a function $f: \omega_{1}\rightarrow \omega$ such that $$V=\Delta_{\mathbf{G}_{1}}\cup\bigcup_{\alpha\in\omega_{1}}C_{\alpha}^{f(\alpha)}\times C_{\alpha}^{f(\alpha)}\subset U.$$ For each $\alpha<\omega_{1}$, put $$A_{\alpha}^{\prime}=\{n\in A_{\alpha}: n\geq f(\alpha)\}$$ and $$B_{\alpha}^{\prime}=\{n\in B_{\alpha}: n\geq f(\alpha)\}.$$ By the condition (b) of the families $\mathscr{A}$ and $\mathscr{B}$, it is easy to see that there exist $\alpha, \beta\in\omega_{1}$ such that $A_{\alpha}^{\prime}\cap B_{\beta}^{\prime}\neq\emptyset$. So, choose $n\in A_{\alpha}^{\prime}\cap B_{\beta}^{\prime}$. Then both $(c(\alpha, n), c_{\alpha})$ and $(c(\beta, n), c_{\beta})$ belong to $V$, thus $$c(\alpha, n)-c_{\alpha}+c(\beta, n)-c_{\beta}\in W_{2}(V)\subset W_{2}(U),$$ which shows $W_{2}(U)\cap X_{n}\neq\emptyset$ (and thus $W_{2}(U)\cap \bigcup_{n\in\mathbb{N}}X_{n}\neq\emptyset$).
\end{proof}

\begin{lemma}\label{c1}
Let $X$ be a space. For each $n\in\mathbb{N}$, the subspace $A_{2^{n}-1}(X)$ of $A(X)$ contains a closed copy of $X^{n}$.
\end{lemma}

\begin{proof}
Let the mapping $f: X^{n}\rightarrow A(X)$ defined by $$f(x_{1},\cdots,  x_{n})=x_{1}+2x_{2}+\cdots +2^{n-1}x_{n}$$ for each $(x_{1}, \cdots, x_{n})\in X^{n}$. It follows from the proof of \cite[Corollary 7.1.16]{AT2008} that $f$ is a homeomorphic mapping from $X^{n}$ onto $f(X^{n})$.
Then $A_{2^{n}-1}(X)$ contains a closed copy of $X^{n}$.
\end{proof}

Now we can prove the following one of the main results in this paper.

\begin{theorem}\label{k-space-characterization}
Let $X$ be a non-metrizable stratifiable $k$-space with a compact-countable $k$-network. Then the following statements are equivalent:
\begin{enumerate}
\item $A(X)$ is a sequential space;
\item $A(X)$ is a $k_{R}$-space;
\item each $A_n(X)$ is a $k_{R}$-space;
\item $A_4(X)$ is a $k_{R}$-space;
\item the space $X$ is the topological sum of a $k_{\omega}$-space and a discrete space.
\end{enumerate}
\end{theorem}

\begin{proof}
The implications of (1) $\Rightarrow$ (2) and (3) $\Rightarrow$ (4) are trivial. Since $X$ is stratifiable, $A(X)$ is a stratifiable space by \cite{S1993}.  By \cite[Proposition 5.10]{BG2014}, the implication of (2) $\Rightarrow$ (3) holds. It suffices to prove (5) $\Rightarrow$ (1) and (4) $\Rightarrow$ (5).

\smallskip
(5) $\Rightarrow$ (1). Let $X=Y\bigoplus D$, where $Y$ is a $k_{\omega}$-space and $D$ a discrete space. It is well-known that $A(X)$ is topologically isomorphic to $A(Y)\times A(D)$. Since $A(Y)$ is a $k_{\omega}$-space by \cite[Theorem 7.4.1]{AT2008} and $A(D)$ is a discrete space, it follows that $A(X)$ is a $k$-space (and thus a sequential space).

\smallskip
(4) $\Rightarrow$ (5). First, we show the following claim.

\smallskip
{\bf Claim 1}: The subspace $\mbox{NI}(X)$ is $\omega_1$-compact\footnote{Recall that a space is called $\omega_1$-{\it compac} if every uncountable subset of $X$ has a cluster point.}.

\smallskip
If not, then there exists a closed, discrete and uncountable subset $\{x_\alpha: \alpha<\omega_1\}$ in $\mbox{NI}(X)$. Since $X$ is paracompact and $\mbox{NI}(X)$ is closed in $X$, there is an uncountable and discrete collection of open subsets $\{U_\alpha: \alpha<\omega_1\}$ in $X$ such that $x_\alpha\in U_\alpha$ for each $\alpha<\omega_1$. For each $\alpha<\omega_1$, since $X$ is sequential and $X\setminus\{x_{\alpha}\}$ is not sequentially closed, there exists a non-trivial sequence $\{x(n, \alpha): n\in \mathbb{N}\}$ converging to $x_\alpha$ in $X$. For each $\alpha<\omega_1$, let $$C_\alpha=\{x(n, \alpha): n\in \mathbb{N}\}\cup \{x_\alpha\}$$ and put $$Z=\bigcup\{C_\alpha: \alpha<\omega_1\}.$$ Obviously, $Z$ is homeomorphic to $\mathbf{G}_{1}$. Without loss of generality, we may assume that $Z=\mathbf{G}_{1}$. Since $\mathbf{G}_{1}$ is a closed subset of $X$ and $X$ is a La\v{s}nev space, it follows from \cite{U1991} that the subspace $A_4(\mathbf{G}_{1})$ is homeomorphic to a closed subset of $A_4(X)$, thus $A_4(\mathbf{G}_{1})$ is a $k_{R}$-subspace. However, by Proposition~\ref{M3}, the subspace $A_4(\mathbf{G}_{1})$ is not a $k_{R}$-subspace, which is a contradiction. Therefore, Claim 1 holds.

By the stratifiability of $X$, each compact subset of $X$ is metrizable \cite{Gr1984}. Then it follows from Theorem~\ref{La} and Lemma~\ref{c1} that $X$ is the topological sum of a family of $k_{\omega}$-spaces. Let $X=\bigoplus_{\alpha\in A} X_{\alpha}$, where each $X_{\alpha}$ is a $k_{\omega}$-space. Let $$A^{\prime}=\{\alpha\in A: X_{\alpha}\ \mbox{is non-discrete}\}.$$ By Claim 1, the set $A^{\prime}$ is countable, hence $X$ is the topological sum of a $k_{\omega}$-space and a discrete space.
\end{proof}

\begin{remark}
The space $X$ is a non-metrizable space in Theorem~\ref{k-space-characterization}. It is natural to ask what happen when $X$ is a metrizable space. Now we give an answer to this question, see Theorem~\ref{A4} below. First we generalize a result of Yamada in \cite{Y1993}, where Yamada proved that $A_{3}(\mathbf{G}_{0})$ is not a $k$-space. Indeed, we prove that $A_{3}(\mathbf{G}_{0})$ is not a $k_{R}$-space.
\end{remark}

\begin{proposition}\label{M2}
The subspace $A_{3}(\mathbf{G}_{0})$ is not a $k_{R}$-space.
\end{proposition}

\begin{proof}
Let $\mathbf{G}_{0}=\bigoplus\{C_{i}: i\in\mathbb{N}\}\bigoplus P_{0}$, where $C_{i}=\{c(i, m): m\in\mathbb{N}\}\cup\{c_{i}\}$ is a convergent sequence with the limit point $c_{_{i}}$ for each $i\in\mathbb{N}$, and let $\{y_{0}\}\cup\{y(n, m): n, m\in\mathbb{N}\}$ be a closed copy of $P_{0}$ in $\mathbf{G}_{0}$, where the set $\{y(n, m): m\in\mathbb{N}\}$ is discrete and open in $\mathbf{G}_{0}$ for each $n\in\mathbb{N}$. Since $A_{3}(\mathbf{G}_{0})$ is a normal $\aleph_{0}$-space, it follows from \cite[Proposition 2.9.2]{B2016} that each compact-finite family is strongly compact-finite, hence by the normality the compact-finite family is also strictly compact-finite. Therefore, it suffices to prove that $A_{3}(\mathbf{G}_{0})$ contains a Cld$^{\omega}$-fan.

\smallskip
For each $n\in\mathbb{N}$, put $$F_{n}=\{c_{n}-c(n, m)+y(n, m): m\in\mathbb{N}\}.$$
We claim that the family $\{F_{n}\}$ is a Cld$^{\omega}$-fan in $A_{3}(\mathbf{G}_{0})$. We divide the proof into the following statements.

\smallskip
(1) Each $F_{n}$ is closed in $A_{3}(\mathbf{G}_{0})$.

\smallskip
Fix an arbitrary $n\in\mathbb{N}$. Indeed, let $$X_{n}=\mbox{supp}(F_{n})=C_{n}\cup\{y(n, m): m\in\mathbb{N}\}.$$ Clearly, $X_{n}$ is closed in $\mathbf{G}_{0}$, thus $A(X_{n})$ is topologically isomorphic to a closed subgroup of $A(\mathbf{G}_{0})$. Hence $A_{3}(X_{n})$  is closed in $A_{3}(\mathbf{G}_{0})$. Since $F_{n}\subset A_{3}(X_{n})$, it suffices to prove that $F_{n}$ is closed in $A_{3}(X_{n})$. By \cite[Theorem 4.5]{Y1998}, $A_{3}(X_{n})$ is metrizable.  Assume to the contrary that there exists a $g\in A_{3}(X_{n})$ such that $g\in \overline{F_{n}}^{A_{3}(X_{n})}\setminus F_{n}$, then there exists a sequence $\{c_{n}-c(n, m_{k})+y(n, m_{k})\}$ in $F_{n}$ converging to $g$. Since $X_{n}$ is paracompact, the closure of the set $$\mbox{supp}(\{g\}\cup\{c_{n}-c(n, m_{k})+y(n, m_{k}): k\in\mathbb{N}\})$$ is compact. However, the set $$\mbox{supp}(\{g\}\cup\{c_{n}-c(n, m_{k})+y(n, m_{k}): k\in\mathbb{N}\})$$ contains a closed infinite discrete subset $\{y(n, m_{k}): k\in\mathbb{N}\}$, which is a contradiction. Therefore, $F_{n}$ is closed in $A_{3}(X_{n})$ (and thus closed in $A_{3}(\mathbf{G}_{0})$).

\smallskip
(2) The family $\{F_{n}\}$ is compact-finite in $A_{3}(\mathbf{G}_{0})$.

\smallskip
Assume to the contrary that there exist a compact subset $K$ in $A_{3}(\mathbf{G}_{0})$ and an increasing sequence $\{n_{i}: i\in\mathbb{N}\}$ in $\mathbb{N}$ such that $K\cap F_{n_{i}}\neq\emptyset$ for each $i\in\mathbb{N}$. For each $i\in\mathbb{N}$, choose a point $c_{n_{i}}-c(n_{i}, m(i))+y(n_{i}, m(i))$. Moreover, since $A_{3}(\mathbf{G}_{0})$ is paracompact, the closure of the set $\mbox{supp}(K)$ is compact. However, the set $\{c_{n_{i}}: i\in\mathbb{N}\}\subset\mbox{supp}(K)$, which is a contradiction.

\smallskip
(3)  The family $\{F_{n}\}$ is not locally finite at the point $y_{0}$ in $A_{3}(\mathbf{G}_{0})$.

\smallskip
Clearly, it suffices to prove that $y_{0}\in \overline{\bigcup_{n\in\mathbb{N}} F_{n}}\setminus\bigcup_{n\in\mathbb{N}} F_{n}$. Indeed, this was proved in \cite[Theorem 3.4]{Y1993}.
\end{proof}

\begin{theorem}\label{A4}
If $X$ is a metrizable space, then the following statements are equivalent:
\begin{enumerate}
\item $A_{n}(X)$ is a $k$-space for each $n\in\mathbb{N}$;
\item $A_{4}(X)$ is a $k$-space;
\item each $A_n(X)$ is a $k_{R}$-space;
\item $A_4(X)$ is a $k_{R}$-space;
\item either $X$ is locally compact and the set $\mbox{NI}(X)$ is separable, or $\mbox{NI}(X)$ is compact.
\end{enumerate}
\end{theorem}

\begin{proof}
The equivalences of (1), (2) and (5) were proved in \cite[Theorem 4.2]{Y1993}. Clearly, (1) $\Rightarrow$ (3) and (3) $\Rightarrow$ (4). It suffices to prove (4) $\Rightarrow$ (5).

Assume that $A_4(X)$ is a $k_{R}$-space. By Claim 1 of the proof in Theorem~\ref{k-space-characterization}, we see that $\mbox{NI}(X)$ is separable. Assume that $X$ is not locally compact and $\mbox{NI}(X)$ is not compact. Then we can take an infinite discrete sequence $\{c_{n}\in \mbox{NI}(X): n\in\mathbb{N}\}$ in $X$. For each $n\in\mathbb{N}$, since $c_{n}\in \mbox{NI}(X)$, there exists a convergent sequence $\{c(n, m): m\in\mathbb{N}\}$ in $X$ which converges to $c_{n}$, and put $$C_{n}=\{c(n, m): m\in\mathbb{N}\}\cup\{c_{n}\}.$$ Moreover, since $X$ is not locally compact, there exists a closed copy of $P_{0}$ in $X$. Let $$\{y_{0}\}\cup\{y(n, m): n, m\in\mathbb{N}\}$$ be a closed copy of $P_{0}$ in $X$, where the set $\{y(n, m): m\in\mathbb{N}\}$ is discrete and open in $\mathbf{G}_{0}$ for each $n\in\mathbb{N}$. Without loss of generality, we may assume that the collection $$\{\{y_{0}\}\cup\{y(n, m): n, m\in\mathbb{N}\}\}\cup\{C_{n}: n\in\mathbb{N}\}$$ is discrete in $X$. Let $$Y=\{y_{0}\}\cup\{y(n, m): n, m\in\mathbb{N}\}\cup\bigcup\{C_{n}: n\in\mathbb{N}\}.$$ Then $Y$ is homeomorphic to $\mathbf{G}_{0}$. Hence, by Proposition~\ref{M2}, $A_{3}(\mathbf{G}_{0})$ is not a $k_{R}$-space, which shows that $A_{3}(Y)$ is not a $k_{R}$-space. Since $Y$ is closed in $X$ and $X$ is metrizable, $A_{3}(Y)$ is embedded into $A_{3}(X)$ as a closed subspace. Since $A(X)$ is stratifiable, it follows from \cite[Proposition 5.10]{BG2014} that $A_{3}(X)$ is not a $k_{R}$-space. Then, by the same fact, $A_{4}(X)$ is not a $k_{R}$-space, which is a contradiction.
\end{proof}

By the proof of Theorem~\ref{A4}, we have the following theorem.

\begin{theorem}\label{A3}
If $X$ is a metrizable space, then the following statements are equivalent:
\begin{enumerate}
\item $A_{3}(X)$ is a $k$-space;
\item $A_{3}(X)$ is a $k_{R}$-space;
\item either $X$ is locally compact or $\mbox{NI}(X)$ is compact.
\end{enumerate}
\end{theorem}

\begin{proof}
The equivalence of (1) and (3) was proved in \cite[Theorem 4.9]{Y1993}, and (1) $\Rightarrow$ (2) is obvious. The proof of Theorem~\ref{A4} implies (2) $\Rightarrow$ (3).
\end{proof}

For each space $X$, the subspace $A_{3}(X)$ contains a closed copy of $X\times X$ by Lemma~\ref{c1}, hence it follows from Theorems~\ref{La} and ~\ref{A3} that we have the following corollary.

\begin{corollary}\label{c11}
Let $X$ be a stratifiable $k$-space with a compact-countable $k$-network. If $A_{3}(X)$ is a $k_{R}$-space, then $X$ satisfies one of the following (a)-(c):

(a) $X$ is a locally compact metrizable space;

(b) $X$ is a metrizable space with the set of all non-isolated points being compact;

(c) $X$ is the topological sum of $k_\omega$-subspaces.
\end{corollary}

\begin{theorem}\label{bw1}
Assume $\mathfrak b=\omega_1$. Let $X$ be a stratifiable $k$-space with a compact-countable $k$-network. If $A_{3}(X)$ is a $k_{R}$-space, then $A_{3}(X)$ is a $k$-space.
\end{theorem}

\begin{proof}
Since $A_{3}(X)$ is a $k_{R}$-space, it follows from Lemma~\ref{c1} that $X^{2}$ is a $k_{R}$-space. By Theorem~\ref{La}, either $X$ is metrizable or $X$ is the topological sum of $k_\omega$-subspaces. If $X$ is a metrizable space, then it follows from Theorem~\ref{A3} that $A_{3}(X)$ is a $k$-space. Now we may assume that $X$ is a non-metrizable space being the topological sum of $k_\omega$-subspaces. Moreover, by the assumption of $\mathfrak b=\omega_1$, there exists a collection $\{f_\alpha\in ^\omega\omega: \alpha<\omega_1\}$ such that if $f\in ^{\omega}\omega$, then there exists $\alpha<\omega_1$ with $f_\alpha(n)>f(n)$ for infinitely may $n\in \omega$. Now we shall prove that $A_{3}(X)$ is a $k$-space.

Indeed, it suffices to prove that NI($X$) is $\omega_{1}$-compact. Assume to the contrary that the subspace $\mbox{NI}(X)$ is not $\omega_1$-compact. Since $X$ is sequential, we can see that $X$ contains a closed copy of $\mathbf{G}_{1}=\bigoplus\{C_\alpha: \alpha<\omega_1\}$, where for each $\alpha\in\omega_1$ the set $$C_\alpha=\{c(\alpha, n): n\in \omega\}\cup \{c_\alpha\}$$ and $c(\alpha, n)\to c_\alpha$ as $n\rightarrow\infty$. Next we divide the proof into the following two cases.

\smallskip
{\bf Case 1} The space $X$ contains no closed copy of $S_{\omega}$.

\smallskip
If $X$ contains no closed copy of $S_{2}$. Then it follows from \cite[Corollaries 2.13 and 3.10]{Lin1997} and \cite[Lemma 8]{YanLin1999} that $X$ has a point-countable base, thus $X$ is metrizable since a stratifiable space with a point-countable base is metrizable \cite{Gr1984}, which is a contradiction. Therefore, we may assume that $X$ contains a closed copy of $S_{2}$. Put
$$X_{1}=\{\infty\}\cup\bigcup_{n\in\omega}D_{n},$$where $D_{n}=\{x_{n}\}\cup\{x_{n}(m): m\in
\omega\}$ for each $n\in\omega$. For each $n, k\in\omega$, put $$D_{n}^{k}=\{x_{n}\}\cup\{x_{n}(m): m>k\}.$$  We endow $X_{1}$ with a topology as follows: each
$x_{n}(m)$ is isolated; the family $\{D_{n}^{k}\}$ is a neighborhood base at the point $x_{n}$ for each $n\in\omega$;
a basic neighborhood of $\infty$ is $$N(f, F)=\{\infty\}\cup \bigcup\{D_{n}^{f(n)}: n\in\omega-F\},$$ where $f\in ^{\omega}\omega$ and $F\in\mathscr{F}$. Then $X_{1}$ is a closed copy of $S_{2}$. Moreover, without loss of generality, we may assume that $X_{1}\subset X$ and $X_{1} \cap \mathbf{G}_{1}=\emptyset$. Let $X_{2}=X_{1}\cup \mathbf{G}_{1}$.

For arbitrary $n, m\in\omega$ and $\beta\in\omega_{1}$, let $$F_{n, m}=\{x_{n}(m)+c(\alpha, n)-c_\alpha: m\leq f_\alpha(n), \alpha\in \omega_{1}\},$$ $$V_{n, m}=\{x_{n}(m)+c(\alpha, n)-x: m\leq f_\alpha(n), \alpha\in \omega_{1}, x\in C_{\alpha}^{n}\},$$and $$C_{\beta}^{n}=C_{\beta}\setminus\{c(\beta, k): k\leq n\}.$$

Since $X$ is stratifiable, it follows from \cite{S1993} that $A(X)$ is stratifiable. Then it follows from
\cite[Proposition 5.10]{BG2014} and \cite{U1991} that $A_{3}(X_{2})$ is a $k_{R}$-subspace. However, we shall claim that the family $\{F_{n, m}\}$ is a  Cld$^{\omega}$-fan in $A_{3}(X_{2})$ and the family $\{V_{n, m}\}$ is compact-finite in $A_{3}(X_{2})$; then since $A_{3}(X_{2})$ is normal, it follows that the family $\{F_{n, m}\}$ is a strict Cld$^{\omega}$-fan in $A_{3}(X_{2})$, which is a contradiction. We divide the proof into the following three statements.

\smallskip
(1) Each $F_{n, m}$ is closed in $A_{3}(X_{2})$.

\smallskip
Fix arbitrary $n, m\in\omega$, and let $X_{n, m}=\mbox{supp}(F_{n, m})$. Obviously, $X_{n, m}$ is a closed discrete subspace of $X_{2}$. By an argument similar to the proof of (2) in Proposition~\ref{M3}, $F_{n, m}$ is closed in $A_{3}(X_{2})$.

\smallskip
(2) The family $\{V_{n, m}\}$ is compact-finite in $A_{3}(X_{2})$.

\smallskip
Since the proof is similar to (1) in Proposition~\ref{M3}, we omit it.

\smallskip
(3) The family $\{F_{n, m}\}$ is not locally finite in $A_{3}(X_{2})$.

\smallskip
It suffices to prove that the family $\{F_{n, m}\}$ is not locally finite at the point $\infty$ in $A_{3}(X_{2})$. We give a uniform base $\mathcal{U}$ of the universal uniformity on $X_{2}$ as follows.
For each $\alpha<\omega_1$ and $n, k\in\omega$, let $$W_{k, n}=(D_{n}^{k}\times D_{n}^{k})\cup\Delta_{x_{n}} \mbox{and}\ U_{k, \alpha}=(C_{\alpha}^{k}\times C_{\alpha}^{k})\cup\Delta_\alpha,$$ where $\Delta_{x_{n}}$ and $\Delta_\alpha$ are the diagonals of $D_n\times D_n$ and $C_\alpha\times C_\alpha$ respectively. For each $f\in ^\omega\omega$, $g\in ^{\omega_1}\omega$ and $F\in\mathscr{F}$, let $$U(g, f, F) =\bigcup\{U_{g(\alpha), \alpha}: \alpha<\omega_1\}\cup (N(f, F)\times N(f, F))\cup (\bigcup_{n\in\omega}W_{f(n), n}\times W_{f(n), n})\cup \Delta_{X_{2}}.$$ Put $$\mathscr{U}=\{U(g, f, F): g\in ^{\omega_1}\omega, f\in ^\omega\omega, F\in\mathscr{F}\}.$$
Then the family $\mathscr{U}$ is a uniform base of the universal uniformity on the space $X_{2}$. Put $$\mathcal{W}=\{W(P): P\in\mathscr{U}^{\omega}\}.$$ Then it follows from \cite{Y1993} that $\mathcal{W}$ is a neighborhood base at 0 in $A(X_{2})$.

Next we prove that $\infty\in\overline{H}\setminus H$ in $A_{3}(X_{2})$, where $H=\bigcup_{n, m\in\omega}F_{n, m}$. Obviously, the family $\{(\infty+U)\cap A_3(X_{2}): U\in \mathcal{W}\}$ is a neighborhood base at $\infty$ in $A_3(X_{2})$.
We shall prove $(\infty+U)\cap A_3(X_{2})\cap H\neq \emptyset$ for each $U\in \mathcal{W}$, which implies $\infty\in \overline{H}\setminus H$ in $A_3(X_{2})$. Fix an $U\in\mathcal{W}$.  Then there exist a sequence $\{h_{i}\}_{i\in\omega}$ in $^{\omega_1}\omega$, a sequence $\{g_{i}\}_{i\in\omega}$ in $^{\omega}\omega$  and a sequence $\{F_{i}\}_{i\in\omega}$ in $\mathscr{F}$ such that $$U=\{x_1-y_1+x_2-y_2+...+x_n-y_n: (x_i, y_i)\in U(h_i, g_i, F_{i}), i\leq n, n\in \omega\}.$$ Let $$B=\{x'-\infty+x''-y'': (x', \infty)\in N(g_1, F_{1})\times N(g_1, F_{1}), (x'', y'')\in U_{h_{1}(\alpha), \alpha}, \alpha<\omega_1\}.$$ Then $B\subset U$. By the assumption, there exists $\alpha< \omega_1$ such that $f_\alpha (k)\geq g_1(k)$ for infinitely many $k$. Pick a $k'>h_{1}(\alpha)$ such that $k'\not\in F_{1}$ and $(x_{k'} (f_\alpha(k')), \infty)\in N(g_1, F_{1})\times N(g_1, F_{1})$. Then $$x_{k'} (f_\alpha(k'))-\infty+c(\alpha, k')-c_\alpha\in U,$$ hence
\begin{eqnarray}
\infty+x_{k'} (f_\alpha(k'))-\infty+c(\alpha, k')-c_{\alpha}&=&x_{k'} (f_\alpha(k'))+c(\alpha, k')-c_\alpha\nonumber\\
&\in&((\infty+U)\cap A_3(Z))\cap F_{k', f_\alpha(k')}\nonumber\\
&\subset&F_{k', f_\alpha(k')}.\nonumber
\end{eqnarray}

\smallskip
{\bf Case 2} The space $X$ contains a closed copy of $S_{\omega}$.

\smallskip
Then $X$ contains a closed subspace $Y$ which is homeomorphic to $S_\omega$. Moreover, without loss of generality, we may assume that $Y \cap \mathbf{G}_{1}=\emptyset$. Let $Z=Y\cup \mathbf{G}_{1}$.
For arbitrary $n, m\in\omega$, let $$H_{n, m}=\{a(n, m)+c(\alpha, n)-c_\alpha: m\leq f_\alpha(n), \alpha\in \omega_{1}\}.$$

By an argument similar to the proof of Case 1, we can prove that the family $\{H_{n, m}\}$ is a strict Cld$^{\omega}$-fan in the $k_{R}$-subspace $A_{3}(Z)$, which is a contradiction.
\end{proof}

\begin{note}
 By the above proof, it follows that if the space $X$ in Theorem~\ref{bw1} is non-metrizable then $A(X)$ is a $k$-space.
 \end{note}
\maketitle
\section{Open questions}
In this section, we pose some open questions about $k_{R}$-spaces.

In \cite{GMT1}, the authors proved that each countably compact $k$-space with a point-countable $k$-network is compact and metrizable. Therefore, we have the following question.

\begin{question}
Is each countably compact $k_{R}$-space with a point-countable $k$-network metrizable?
\end{question}

In \cite{BG2014}, the authors proved that each closed subspace of a stratifiable $k_{R}$-space is a $k_{R}$-subspace. However, the following question is still open. First, we recall a concept.

 \begin{definition}
A topological space $X$ is a {\it $k$-semistratifiable space} if for each open subset $U$ in $X$, one can assign a sequence $\{U_{n}\}_{n=1}^{\infty}$ of closed subsets in $X$ such that

(a) $U=\bigcup_{n\in\mathbb{N}}U_{n}$;

(b) for each compact subset $K$ with $K\subset U$ there exists $n\in\mathbb{N}$ such that $K\subset U_{n}$;

(c) $U_{n}\subset V_{n}$ whenever $U\subset V$.
\end{definition}

\begin{question}
Is each closed subspace of $k$-semistratifiable $k_{R}$-space a $k_{R}$-subspace?
\end{question}

Moreover, we have the following question:

\begin{question}
Is each closed subgroup of a $k_{R}$-topological group (free Abelian topological group) $k_{R}$?
\end{question}

In \cite{B2016}, the author posed the following question:

\begin{problem}\cite[Problem 3.5.5]{B2016}
Assume that a Tychonoff space contains no Cld-fan. Is $X$ a $k$-space?
\end{problem}

Indeed, the following question is still open.

\begin{question}
Assume that a topological group $G$ contains no Cld-fan. Is $G$ a $k$-space?
\end{question}

Furthermore we have the following question.

\begin{question}
Assume that a topological group $G$ contains no strict Cld-fan. Is $G$ a $k_{R}$-space?
\end{question}

By Theorem~\ref{kr charac}, we have the following question.

\begin{question}
Can we replace ``$k$-space'' with ``$k_{R}$-space'' in the conditions (1), (2) and (3) in Theorem~\ref{kr charac}?
\end{question}

We conjecture the following question is positive.

\begin{question}\label{q2}
Let $X$ and $Y$ be two sequential spaces. If $X\times Y$ is a $k_{R}$-space, is $X\times Y$ a $k$-space?
\end{question}

By Theorem~\ref{bw1}, it is natural to pose the following question.

\begin{question}
Assume $\mathfrak b>\omega_1$. Let $X$ be a stratifiable $k$-space with a compact-countable $k$-network. If $A_{3}(X)$ is a $k_{R}$-space, then is $A_{3}(X)$ a $k$-space?
\end{question}

It is well-known that neither $(S_{2})^{\omega}$ nor $(S_{\omega})^{\omega}$ are $k$-spaces. Hence it is natural to pose the following question:

\begin{question}
Are $(S_{2})^{\omega}$ and $(S_{\omega})^{\omega}$ $k_{R}$-spaces?
\end{question}

{\bf Acknowledgements}. The authors are thankful to the
anonymous referees, as well as to professor Jiling Cao, for valuable remarks and corrections and all other sort of help related to the content of this article.
In particular, the anonymous referees point-out an incorrect proof in Lemma~\ref{M1} in the preliminary version of our paper. 



\begin{thebibliography}{99}
\bibitem{AOP1989} A.V. Arhangel'ski\v\i, O.G. Okunev, V.G. Pestov,  {\it
  Free topological groups over metrizable spaces}, Topology Appl., {\bf 33}(1989),
  63--76.

\bibitem{AT2008} A. Arhangel'ski\v{\i}, M. Tkachenko,
  {\it Topological groups and related structures}, Atlantis Press, Paris; World
  Scientific Publishing Co. Pte. Ltd., Hackensack, NJ, 2008.

 \bibitem{B2016}  T. Banakh, {\it Fans and their applications in General Topology, Functional Analysis and Topological Algebra}, arXiv:1602.04857v2.

 \bibitem{BG2014}  T. Banakh, S. Gabriyelyan, {\it On the $C_k$-stable closure of the class of (separable) metrizable spaces}, Monatshefte f\"{u}r Mathematik, {\bf 180}(2016), 39--64.

  \bibitem{Bl1977} J.L. Blasco, {\it On $\mu$-Spaces and $k_R$-Spaces}, Proc. Amer. Math. Soc., {\bf 67(1)}(1977), 179--186.

 \bibitem{B1966} C.R. Borges, {\it On stratifiable spaces}, Pacific J. Math., {\bf 17}(1966),  1--16.

 \bibitem{B1981} C.R. Borges, {\it A stratifiable $k_R$-space which is not a $k$-space}, Proc. Amer. Math. Soc., {\bf 81}(1981), 308--310.

\bibitem{C1961} J.G. Ceder, {\it Some generalizations of metric spaces}, Pacific J. Math., {\bf 11}(1961), 105-¨C125.

  \bibitem{E1989} R. Engelking, {\it General Topology} (revised and completed edition),
  Heldermann Verlag, Berlin, 1989.

 \bibitem{F1985} L. Foged, {\it A characterization of closed images of metric spaces}, Proc. Amer. Math. Soc., {\bf 95}(1985), 487--490.

 \bibitem{Gr1980} G. Gruenhage, {\it $k$-spaces and products of closed image of metric spaces}, Proc. Amer. Math. Soc., {\bf 80(3)}(1980), 478--482.

 \bibitem{Gr1984} G. Gruenhage, {\it Generalized metric spaces}, In: K. Kunen, J. E.
  Vaughan(Eds.), Handbook of Set-Theoretic Topology, Elsevier Science
  Publishers B.V., Amsterdam, 1984, 423--501.

\bibitem{GMT1} G. Gruenhage, E.A Michael, Y. Tanaka, {\it Spaces determined by point-countable covvers},
Pacific J. Math., {\bf 113}(1984), 303--332.

 \bibitem{H1964} R.W. Heath, {\it Screenability, pointwise paracompactness and metrization of Moore
spaces}, Canad. J. Math., {\bf 16}(1964), 763--770.

\bibitem{JY1992} H.J.K. Junnila, Z. Yun, {\it $\aleph$-spaces and spaces with a $\sigma$-hereditarily
closure-preserving $k$-network}, Topology Appl., {\bf 44}(1992), 209--215.

 \bibitem{JW1997} W. Just, M. Weese, {\it Discovering Modern Set Theorey II}, Graduate Studies in Mathematics, 18, 1997.

  \bibitem{K1975} J. Kelley, General Topology, Springer-Verlag, New York-Berlin, 1975.

 \bibitem{L2006} J. Li, {\it On $k$-spaces and $k_{R}$-spaces}, Czech. Math. J., {\bf 55}(2006), 941--945.

 \bibitem{LL2016} F. Lin, C. Liu, {\it The $k$-spaces property of free Abelian topological groups over non-metrizable La\v{s}nev spaces}, submitted.

 \bibitem{LLL2016} F. Lin, S. Lin, C. Liu, {\it The $k_{R}$-property in free topological groups}, Indagationes Mathematicae, {\bf 28}(2017), 1056--1066.

 \bibitem{Lin2015} S. Lin, {\it Point-countable Covers and Sequence-covering Mappings} (second edition), Beijing,
China Science Press, 2015.

 \bibitem{Lin1991} S. Lin, {\it Note on $k_{R}$-spaces}, Q \& A in General Topology, {\bf 9}(1991), 227--236.

 \bibitem{Lin1997} S. Lin, {\it A note on the Aren's space and sequential fan}, Topology Appl., {\bf 81(3)}(1997), 185--196.

 \bibitem{LiuLin2006} S. Lin, C. Liu, {\it $k$-space property of $S_{\omega}\times X$ and related results}, Acta Math. Sinica (Chinese Seeries), {\bf 49(1)}(2006), 29--38.

 \bibitem{Liu1992} C. Liu, {\it Spaces with a $\sigma$-compact finite $k$-network.}, Questions Answers in
General Topology, {\bf 10}(1992), 81--87.

 \bibitem{LiuLin1997} C. Liu, S. Lin, {\it $k$-spaces property of product spaces}, Acta Math. Sinica (New Seeries), {\bf 13(4)}(1997), 537--544.

 \bibitem{LiuT1998} C. Liu, Y. Tanaka, {\it Star-countable $k$-networks, compact-countable $k$-networks, and related results}, Houston J. Math., {\bf 24(4)}(1998), 655--670.

  \bibitem{M1973} E. Michael,, {\it On $k$-spaces, $k_{R}$-spaces and $k(X)$}, Pacific J. Math., {\bf 47}(1973), 487--498.

\bibitem{No1970} N. Noble, {\it The continuity of functions on Cartesian products}, Trans. Amer. Math. Soc., {\bf 149}(1970), 187-¨C198.

\bibitem{Ro1996} D.J.F. Robinson,
  {\it A course in the theory of groups}, Springer-Verlag, Berlin, 1996.

  \bibitem{Sa97a} M. Sakai, {\it On spaces with a star-countable $k$-network},
{Houston J. Math.}, {\bf 23}(1)(1997), 45--56.

 \bibitem{S1993} O.V. Sipacheva,  {\it On the stratifiability of free Abelian topological groups}, Topology Proc., {\bf 18}(1993), 271--311.

 \bibitem{T1976}  Y. Tanaka, {\it A characterization for the products of $k$-and $\aleph_0$-spaces and related results},
{Proc. Amer. Math. Soc.}, {\bf 59}(1)(1976), 149--155.

\bibitem{U1991} V. Uspenskii, {\it Free topological groups of metrizable spaces}, Math. USSR Izv. {\bf 37}(1991), 657--680.

\bibitem{Y1993} K. Yamada, {\it Characterizations of a metrizable space such that every $A_n(X)$ is a $k$-space}, Topology Appl., {\bf 49}(1993), 74--94.

  \bibitem{Y1997} K. Yamada, {\it Tightness of free abelian topological groups and of
finite products of sequential fans}, Topology Proc., {\bf 22}(1997), 363--381.

  \bibitem{Y1998} K. Yamada, {\it Metrizable subspaces of free topological groups
  on metrizable spaces}, Topology Proc., {\bf 23}(1998), 379--409.

  \bibitem{YanLin1999} P.F. Yan, S. Lin, {\it Point-countable $k$-networks, cs$^{\ast}$-networks and $\alpha_{4}$-spaces}, Topology Proc., {\bf 24}(1999), 345--354.


  \end{thebibliography}
  \end{document}